\documentclass{article}
\usepackage{fullpage}
\usepackage[utf8]{inputenc}
\usepackage{subfigure}
\usepackage[round]{natbib}
\usepackage{amsfonts,amssymb,amsmath,dsfont,amsthm}
\usepackage{enumitem}
\usepackage{xspace}
\usepackage{graphicx}
\usepackage[colorlinks=true,citecolor=blue]{hyperref}
\usepackage{bm}
\usepackage{comment}

\newtheorem{proposition}{Proposition}
\newtheorem{lemma}{Lemma}

\newcommand{\proba}{\mathds{P}}
\newcommand{\mean}{\mathds{E}}
\newcommand{\var}{\mathds{V}}

\newcommand{\smallo}{o}
\newcommand{\integers}{\mathds{Z}}
\newcommand{\reals}{\mathds{R}}

\newcommand{\hessian}{\operatorname{Hessian}}

\newcommand{\mC}{\mathcal{C}}

\newcommand{\mV}{\mathcal{V}}

\newcommand{\vect}[1]{\bm{#1}}
\newcommand{\vk}{\vect{k}}
\newcommand{\vtheta}{\vect{\theta}}

\newcommand*{\eg}{\textit{e.g.}\@\xspace}

\author{\'Elie de Panafieu\thanks{\href{mailto:depanafieuelie@gmail.com}{depanafieuelie@gmail.com}}, François Durand\thanks{\href{mailto:fradurand@gmail.com}{fradurand@gmail.com}}\\
{\normalsize Nokia Bell Labs France}}

\title{Limit Distribution of Two Skellam Distributions,\\Conditionally on Their Equality}

\begin{document}
\maketitle

This note provides a proof of the following proposition.

\begin{proposition}\label{th:main_result}
Let $A_n$, $B_n$, $C_n$ and $D_n$ denote independent random Poisson variables of respective parameters $n \tau_A$, $n \tau_B$, $n \tau_C$ and $n \tau_D$. We denote $X_n = (A_n - B_n \mid A_n - B_n = C_n - D_n)$.
\begin{enumerate}
\item\label{enumi:degenerateCase} If $\tau_A = \tau_B = 0$, or $\tau_B = \tau_C = 0$, or $\tau_C = \tau_D = 0$, or $\tau_D = \tau_A = 0$, then the distribution of $X_n$ is a Dirac measure in 0.
\item\label{enumi:generalCase} Otherwise,
\begin{align*}
	\mean(X_n) &= n E + E' + \smallo(1), \\
	\var(X_n) &= n V + V' + \smallo(1),
\end{align*}
where
\begin{align*}
E &= \frac{\tau_A \tau_C - \tau_B \tau_D}{\sqrt{(\tau_A + \tau_D)(\tau_B + \tau_C)}}, \\
E' &= - \frac{\tau_A \tau_C - \tau_B \tau_D}{4 (\tau_A + \tau_D)(\tau_B + \tau_C)}, \\
V &= \frac{(\tau_A\tau_C + 2\tau_A\tau_B + \tau_B\tau_D) (\tau_A\tau_C + 2\tau_C\tau_D + \tau_B\tau_D)}
          {2\big[(\tau_A + \tau_D)(\tau_B + \tau_C)\big]^{\frac{3}{2}}}, \\
V' &= - \frac{(\tau_A\tau_C + \tau_B\tau_D)(\tau_A\tau_B + \tau_C\tau_D) + 4\tau_A\tau_B\tau_C\tau_D}
             {4\big[(\tau_A + \tau_D)(\tau_B + \tau_C)\big]^2}.
\end{align*}
And the distribution of $(X_n - n E) / \sqrt{n V}$ is asymptotically Gaussian.
\end{enumerate}
\end{proposition}

Since case \ref{enumi:degenerateCase} is trivial, we only need to prove case \ref{enumi:generalCase} and thus assume that $\tau_A + \tau_B > 0$ and $\tau_B + \tau_C > 0$ and $\tau_C + \tau_D > 0$ and $\tau_D + \tau_A > 0$.

Up to exchanging roles between $A_n$ and $B_n$, and between $C_n$ and $D_n$, we also assume that $\tau_A \tau_C \ge \tau_B \tau_D$. Intuitively, it means that in $X_n = (A_n - B_n \mid A_n - B_n = C_n - D_n)$, the ``positive'' forces $A_n$ and $C_n$ are stronger than the ``negative'' ones, $B_n$ and $D_n$. It corresponds to the cases where $E$, the main term in the asymptotic development of the expectation, will be proved nonnegative.

\section{Generic Case: All Coefficients Are Positive}\label{sec:positive_coefs}

For this section, we add the assumption that $\tau_A$, $\tau_B$, $\tau_C$ and $\tau_D$ are all positive. This assumption will be removed in the second section.

Let $P_n(u)$ denote the probability generating function
\[
    P_n(u) = \sum_{m \in \integers} \proba(A_n - B_n = m \mid A_n - B_n = C_n - D_n) u^m.
\]

Introducing the function
\[
    F_n(u) =
    \sum_{\substack{a, b, c, d \geq 0\\ a - b = c - d}}
    \proba(A_n = a) \proba(B_n = b) \proba(C_n = c) \proba(D_n = d)
    u^{a - b},
\]
we obtain
\[
    P_n(u) = \frac{F_n(u)}{F_n(1)}.
\]

Expressing the probabilities explicitly, the expression becomes
\[
    F_n(u) =
    e^{-n (\tau_A + \tau_B + \tau_C + \tau_D)}
    \sum_{\substack{a, b, c, d \geq 0\\ a - b = c - d}}
    \frac{(n \tau_A)^a}{a!}
    \frac{(n \tau_B)^b}{b!}
    \frac{(n \tau_C)^c}{c!}
    \frac{(n \tau_D)^d}{d!}
    u^{a-b}.
\]

Introducing the variable $g = a - b = c - d$, we obtain
\[
    F_n(u) =
    e^{-n (\tau_A + \tau_B + \tau_C + \tau_D)}
	G_n(u),
\]
where
\[
	G_n(u) = 
    \sum_{\substack{b, d \ge 0 \\ g \ge max(-b, -d)}}
    \frac{(n \tau_A)^{(b + g)}}{(b + g)!}
    \frac{(n \tau_B)^b}{b!}
    \frac{(n \tau_C)^{(d + g)}}{(d + g)!}
    \frac{(n \tau_D)^d}{d!}
    u^g.
\]

The Stirling approximation is introduced and we define the values
$x = \frac{b}{n}$, $y = \frac{d}{n}$, $z = \frac{g}{n}$. Then:
\[
    G_n(u) =  
    \sum_{\substack{b, d \ge 0 \\ g \ge max(-b, -d)}}
    \psi_n(x,y,z)
    e^{- n \phi_u(x,y,z)},
\]
where
\begin{align*}
    \psi_n(x,y,z) =\ &
	\frac{(n (x+z))^{n (x+z)} e^{- n (x+z)}}{(n (x+z))!}
    \frac{(n x)^{n x} e^{- n x}}{(n x)!}
    \frac{(n (y+z))^{n (y+z)} e^{- n (y+z)}}{(n (y+z))!}
    \frac{(n y)^{n y} e^{- n y}}{(n y)!},
    \\
    \phi_u(x,y,z) =\ &
    x \log(x) - x (1 + \log(\tau_B))
    + (x+z) \log(x+z) - (x+z) (1 + \log(\tau_A))
    \\&
    + y \log(y) - y (1 + \log(\tau_D))
    + (y+z) \log(y+z) - (y+z) (1 + \log(\tau_C)) 
    - z \log(u).
\end{align*}
The above expression is well defined because we assumed that $\tau_A$, $\tau_B$, $\tau_C$ and $\tau_D$ are all positive.

We will soon see that
the main contributions to $G_n(u)$
come from the vicinity of the minimum of $\phi_u$.
We will first compute this main contribution,
then prove that the rest of the sum is negligible.

The function $\phi_u$ is convex with a unique minimum. 
Therefore, there is a small enough vicinity $\mV$ of this minimum such that $\phi_u$
is larger anywhere outside this vicinity than anywhere inside.
Without loss of generality, we assume that $\mV$
does not contain the origin.
Thus, uniformly on $\mV$, we have
\[
    \psi_n(x,y,z) \sim
    \frac{\psi(x,y,z)}{(2 \pi n)^2}
    \quad \text{where} \quad
    \psi(x,y,z) := \frac{1}{\sqrt{x (x+z) y (y+z)}}.
\]


We apply the following classical lemma
(Laplace method) 
to extract the asymptotics.

\begin{lemma}[Laplace Method] \label{thm:laplaceMethod}
  Consider a compact set $\mC$ of $\reals^d$ and the series
  \[
  I_n = \sum_{\substack{\vk \in \integers_{\geq 0}^d\\ \vk/n \in \mC}} \psi(\vk / n) e^{- n \phi(\vk / n)}
  \]
  where $\psi, \phi$ are differentiable functions from $\mC$ to $\reals$.
  Assume furthermore that
  $\phi$ has a unique global minimum $\vtheta^*$
    which is located in the interior of $\mC$,
    $\phi$ is three-times differentiable in a neighborhood of $\vtheta^*$,
    $\hessian_{\phi}(\vtheta^*) > 0$ and $\psi(\vtheta^*) \neq 0$.
  Then
  \[
  I_n \sim (2 \pi n)^{d/2} \psi(\vtheta^*) \frac{e^{-n \phi(\vtheta^*)}}{\sqrt{\hessian_{\phi}(\vtheta^*)}}.
  \]
\end{lemma}

\begin{proof}
  There are many variants of this classic result.
  The one-dimensional case ($d = 1$) is treated by \cite{masoero15}.
  There, the approximation of the sum by a Riemann integral is justified.
  The same transformation applies to the multivariate case.
  The asymptotics of the triple integral is then obtained
  by a multivariate Laplace method, see \eg \cite{PW13}.
\end{proof}


Denoting $x^*, y^*, z^*$ the minimal point of $\phi_u$, we conclude that the contribution from $\mV$ to the sum $G_n(u)$ has an asymptotics of the form $e^{- n \phi(x^\star,y^\star,z^\star)}$
multiplied by a polynomial term.
Because $\phi_u$ outside of $\mV$ is larger than
$\phi_u(x^\star,y^\star,z^\star) + \epsilon$
for some positive $\epsilon$,
we conclude that the contribution of the rest of the sum
is exponentially small and therefore negligible in the asymptotics.
(Also: outside of $\mV$, we use Stirling bounds to bound $\psi_n(x,y,z)$).

In order to get the asymptotics of $G_n(u)$, what remains to do is to evaluate $\psi(x^*, y^*, z^*)$, $\phi_u(x^*, y^*, z^*)$ and $\hessian_{\phi_u}(x^*, y^*, z^*)$. The minimal point of $\phi_u$ is characterized by the system
\begin{align*}
	x^* (x^* + z^*) &= \tau_A \tau_B,\\
	y^* (y^* + z^*) &= \tau_C \tau_D,\\
	(x^* + z^*) (y^* + z^*) &= \tau_A \tau_C u.
\end{align*}

The solution is given by:
\begin{align*}
x^* &= \frac{\tau_B}{\sqrt{u}} \sqrt{\frac{ \tau_A \sqrt{u} + \frac{\tau_D}{\sqrt{u}} }{ \tau_C \sqrt{u} + \frac{\tau_B}{\sqrt{u}} }}, \\
y^* &= \frac{\tau_D}{\sqrt{u}} \sqrt{\frac{ \tau_C \sqrt{u} + \frac{\tau_B}{\sqrt{u}} }{ \tau_A \sqrt{u} + \frac{\tau_D}{\sqrt{u}} }}, \\
z^* &= \frac{ \tau_A \tau_C u - \frac{\tau_B \tau_D}{u} }{ \sqrt{ \left(\tau_A \sqrt{u} + \frac{\tau_D}{\sqrt{u}}\right)\left(\tau_C \sqrt{u} + \frac{\tau_B}{\sqrt{u}}\right) } }.
\end{align*}
We have $x^* > 0$ and $y^* > 0$. Moreover, the assumption $\tau_A \tau_C \geq \tau_B \tau_D$ ensures that for $u$ close enough to 1, $z^* > \max(-x^*, -y^*)$. As a consequence, $(x^*, y^*, z^*)$ is on the interior of the integration zone defining $G_n(u)$, which validates the approximation by a Riemann integral mentioned above.

Simple algebra leads to:
\begin{align*}
\psi(x^*, y^*, z^*) &= \frac{1}{\sqrt{\tau_A \tau_B \tau_C \tau_D}}, \\
\phi_u(x^*, y^*, z^*) &= -2 \sqrt{ \left(\tau_A \sqrt{u} + \frac{\tau_D}{\sqrt{u}}\right)\left(\tau_C \sqrt{u} + \frac{\tau_B}{\sqrt{u}}\right) },\\
\hessian_{\phi_u}(x^*, y^*, z^*) &= \frac{2}{\tau_A \tau_B \tau_C \tau_D} \sqrt{ \left(\tau_A \sqrt{u} + \frac{\tau_D}{\sqrt{u}}\right)\left(\tau_C \sqrt{u} + \frac{\tau_B}{\sqrt{u}}\right) }. \\
\end{align*}

Applying Lemma~\ref{thm:laplaceMethod}, we then have:
\[
	G_n(u) \sim 
	\frac{1}{2\sqrt{\pi n}} 
	\gamma(u)^{-\frac{1}{4}}
	\exp \left( 2n \sqrt{\gamma(u)} \right),
\]
where
\[
	\gamma(u) = \left(\tau_A \sqrt{u} + \frac{\tau_D}{\sqrt{u}}\right)\left(\tau_C \sqrt{u} + \frac{\tau_B}{\sqrt{u}}\right).
\]

To obtain the convergence in distribution to a Gaussian law,
we will apply the \emph{Quasi-powers Theorem}, due to~\cite{H98},
which proof is also given by~\cite{FS09} (Lemma~IX.1)
(we use a slightly weaker version because we are not interested into the speed of convergence).

\begin{lemma}[Quasi-powers] \label{th:quasi_powers}
Assume that the Laplace transform $\mean(e^{s X_n})$
of a sequence of random variables $X_n$
is analytic in a neighborhood of $0$,
and has an asymptotics of the form
\[
  \mean(e^{s X_n}) \underset{n \to +\infty}{\sim} e^{\beta_n f(s) + g(s)},
\]
with $\beta_n \to +\infty$ as $n \to +\infty$,
and $f(s)$, $g(s)$ analytic on a neighborhood of $0$.
Assume also the condition $f''(0) \neq 0$.
Under these assumptions, the mean and variance of $X_n$ satisfy
\begin{align*}
  \mean(X_n) &= \beta_n f'(0) + g'(0) + \smallo(1),\\
  \var(X_n) &= \beta_n f''(0) + g''(0) + \smallo(1),
\end{align*}
and the distribution of $(X_n - \beta_n f'(0)) / \sqrt{\beta_n f''(0)}$
is asymptotically Gaussian.
\end{lemma}

We apply Lemma~\ref{th:quasi_powers} to $X_n = (A_n - B_n \ |\ A_n - B_n = C_n - D_n)$. Using the asymptotics of $G_n$, we have:
\[
	\mean(e^{s X_n}) 
	= \frac{F_n(e^s)}{F_n(1)}
	\underset{n \to +\infty}{\sim} \exp \left[
		2 n \left( \sqrt{\gamma(e^s)} - \sqrt{\gamma(1)} \right)
		- \frac{1}{4} \left( \log(\gamma(e^s)) - \log(\gamma(1)) \right)
	\right].
\]

The result of Proposition~\ref{th:main_result}
is then obtained by application of Lemma~\ref{th:quasi_powers}
with
\begin{align*}
	\beta_n &= n,\\
	f(s) &= 2 \left( \sqrt{\gamma(e^s)} - \sqrt{\gamma(1)} \right), \\
	g(s) &= - \frac{1}{4} \left( \log(\gamma(e^s)) - \log(\gamma(1)) \right).
\end{align*}
The assumptions $\tau_A + \tau_B > 0$, $\tau_B + \tau_C > 0$, $\tau_C + \tau_D > 0$ and $\tau_D + \tau_A > 0$ ensure that 
$$
f''(0) = \frac{(\tau_A\tau_C + 2\tau_A\tau_B + \tau_B\tau_D) (\tau_A\tau_C + 2\tau_C\tau_D + \tau_B\tau_D)}
          {2\big[(\tau_A + \tau_D)(\tau_B + \tau_C)\big]^{\frac{3}{2}}}
$$
is positive.

\section{Degenerate Case: Some Coefficients Are Zero}

We now consider the case where one or several coefficients $\tau$ vanish. Considering our assumptions $\tau_A + \tau_B > 0$ and $\tau_B + \tau_C > 0$ and $\tau_C + \tau_D > 0$ and $\tau_D + \tau_A > 0$ and $\tau_A \tau_C \ge \tau_B \tau_D$, there are only two cases, up to symmetries:
\begin{itemize}
\item $\tau_B = 0$ and the other coefficients are positive,
\item $\tau_B = \tau_D = 0$ and the other coefficients are positive.
\end{itemize}

In both cases, the proof is based on the same principle as in the first section. The main difference is that the triple sum is replaced by a double sum in the first case, and by a simple sum in the second case.

\subsection{$\tau_B = 0$ and the other coefficients are positive}

The probability generating function becomes
\[
    P_n(u) = \sum_{m \in \integers} \proba(A_n = m \mid A_n = C_n - D_n) u^m.
\]
Introducing the function
\[
    F_n(u) =
    \sum_{\substack{a, c, d \geq 0\\ a = c - d}}
    \proba(A_n = a) \proba(C_n = c) \proba(D_n = d)
    u^{a},
\]
we obtain
\[
    P_n(u) = \frac{F_n(u)}{F_n(1)}.
\]
Expressing the probabilities explicitly, the expression becomes
\[
    F_n(u) =
    e^{-n (\tau_A + \tau_C + \tau_D)}
    \sum_{\substack{a, c, d \geq 0\\ a = c - d}}
    \frac{(n \tau_A)^a}{a!}
    \frac{(n \tau_C)^c}{c!}
    \frac{(n \tau_D)^d}{d!}
    u^{a}
\]
and we obtain
\[
    F_n(u) =
    e^{-n (\tau_A + \tau_C + \tau_D)}
	G_n(u),
\]
where
\[
    G_n(u) = 
    \sum_{a, d \ge 0}
    \frac{(n \tau_A)^{a}}{a!}
    \frac{(n \tau_C)^{a + d}}{(a+d)!}
    \frac{(n \tau_D)^d}{d!}
    u^a.
\]
The Stirling approximation is introduced and we define the values
$x = \frac{a}{n}$ and $y = \frac{d}{n}$. Then:
\[
    G_n(u) =  
    \sum_{a, d \geq 0}
    \psi_n(x,y)
    e^{- n \phi_u(x,y)},
\]
where
\begin{align*}
    \psi_n(x,y) =\ &
    \frac{(n x)^{n x} e^{- n x}}{(n x)!}
    \frac{(n (x+y))^{n (x+y)} e^{- n (x+y)}}{(n (x+y))!}
    \frac{(n y)^{n y} e^{- n y}}{(n y)!},
    \\
    \phi_u(x,y) =\ &
    x \log(x) - x (1 + \log(\tau_A))
    + (x+y) \log(x+y) - (x+y) (1 + \log(\tau_C))
    \\&
    + y \log(y) - y (1 + \log(\tau_D))
    - x \log(u).
\end{align*}
The minimum of $\phi_u$ is obtained for:
\begin{align*}
x^* &= \frac{\tau_A \tau_C u}{\sqrt{ \tau_A \tau_C u + \tau_C \tau_D }}, \\
y^* &= \frac{\tau_C \tau_D}{\sqrt{ \tau_A \tau_C u + \tau_C \tau_D }}. \\
\end{align*}
This leads to:
\begin{align*}
\psi_n(x^*, y^*) &\sim \frac{1}{(2 \pi n)^\frac{3}{2}} \frac{ (\tau_A \tau_C u + \tau_C \tau_D)^\frac{1}{4} }{\sqrt{ \tau_A {\tau_C}^2 \tau_D u }}, \\
\phi_u(x^*, y^*) &= - 2 \sqrt{\tau_A \tau_C u + \tau_C \tau_D},\\
\hessian_{\phi_u}(x^*, y^*) &= 2 \frac{ \tau_A \tau_C u + \tau_C \tau_D }{ \tau_A {\tau_C}^2 \tau_D u }. \\
\end{align*}
Applying the same reasonning as in the previous section
and Lemma~\ref{thm:laplaceMethod}, we obtain
\[
	G_n(u) \sim 
	\frac{1}{2\sqrt{\pi n}} 
	\gamma(u)^{-\frac{1}{4}}
	\exp \left( 2n \sqrt{\gamma(u)} \right),
\]
where $\gamma(u)$ has the same expression as in Section~\ref{sec:positive_coefs}, applied to the particular case $\tau_B = 0$. From this point, the end of the proof is the same as in Section~\ref{sec:positive_coefs}.

\subsection{$\tau_B = \tau_D = 0$ and the other coefficients are positive}

In this case, we have
\[
    F_n(u) =
    e^{-n (\tau_A + \tau_C)}
    \sum_{\substack{a, c \geq 0\\ a = c}}
    \frac{(n \tau_A)^a}{a!}
    \frac{(n \tau_C)^c}{c!}
    u^{a},
\]
which leads to a simple sum (instead of a double or triple sum):
\[
    G_n(u) = 
    \sum_{a \ge 0}
    \frac{(n \tau_A)^{a}}{a!}
    \frac{(n \tau_C)^{a}}{a!}
    u^a.
\]
As usual, we define the value $x = \frac{a}{n}$ and obtain
\[
    G_n(u) =  
    \sum_{a \geq 0}
    \psi_n(x)
    e^{- n \phi_u(x)},
\]
where
\begin{align*}
    \psi_n(x,y) =\ &
    \left( \frac{(n x)^{n x} e^{- n x}}{(n x)!} \right)^2
    \\
    \phi_u(x,y) =\ &
    x \big( 2 \log(x) - 2 - \log(\tau_A \tau_C u) \big).
\end{align*}
The minimum of $\phi_u$ is obtained for $x^* = \sqrt{\tau_A \tau_C u}$, which leads to:
\begin{align*}
\psi_n(x^*) &\sim \frac{1}{2 \pi n} \frac{ 1 }{\sqrt{ \tau_A \tau_C u }}, \\
\phi_u(x^*) &= - 2 \sqrt{\tau_A \tau_C u},\\
\hessian_{\phi_u}(x^*) &= \frac{2}{ \sqrt{\tau_A \tau_C u} }. \\
\end{align*}
Applying Lemma~\ref{thm:laplaceMethod}, we obtain
\[
	G_n(u) \sim 
	\frac{1}{2\sqrt{\pi n}} 
	\gamma(u)^{-\frac{1}{4}}
	\exp \left( 2n \sqrt{\gamma(u)} \right),
\]
where $\gamma(u)$ has the same expression as in Section~\ref{sec:positive_coefs}, applied to the particular case $\tau_B = \tau_D = 0$. We then conclude like in Section~\ref{sec:positive_coefs}.


\end{document}